\documentclass{amsart}
\pdfoutput=1
\usepackage[pdfauthor={Marton Naszodi and Steven Taschuk},
            pdftitle={On the Transversal Number and VC-Dimension
                      of Families of Positive Homothets of a Convex Body}
           ]{hyperref}
\usepackage{amsfonts}
\usepackage{amsmath}
\usepackage{amssymb}
\usepackage{graphicx}

\newtheorem{thm}{Theorem}
\newtheorem{prop}[thm]{Proposition}
\newtheorem{ex}[thm]{Example}
\newtheorem{cor}[thm]{Corollary}

\renewcommand{\Re}{\mathbb R}
\newcommand{\Ze}{\mathbb Z}
\newcommand{\Ren}{\Re^n}
\newcommand{\N}{\mathbb N}
\newcommand{\FF}{\mathcal F}

\newcommand{\st}{\colon}  

\newcommand{\nlogn}{n\log n+ \log\log n + 5n}

\DeclareMathOperator{\card}{card}
\DeclareMathOperator{\vcdim}{vcdim}
\DeclareMathOperator{\vol}{vol}
\DeclareMathOperator{\cl}{cl}
\DeclareMathOperator{\conv}{conv}

\begin{document}
\title[Transversal Number and VC-Dimension of a Family of Homothets]{On the Transversal Number and VC-Dimension
    of Families of Positive Homothets of a Convex Body}

\author{M\'arton Nasz\'odi}
\author{Steven Taschuk}
\address{Dept.\ of Math.\ and Stats., 632 Central Academic Building,
    University of Alberta, Edmonton, AB, Canada T6G 2G1}
\email{mnaszodi@math.ualberta.ca, staschuk@ualberta.ca}
\thanks{The first named author holds a Postdoctoral Fellowship
    of the Pacific Institute for the Mathematical Sciences
    at the University of Alberta.
    The second naemd author was supported
    by an Alexander Graham Bell Canada Graduate Scholarship
    of the Natural Sciences and Engineering Research Council of Canada.}

\date{2009 July 29}

\begin{abstract}
Let $\FF$ be a family of positive homothets (or translates) of a given convex body~$K$ in $\Ren$.
We investigate two approaches to measuring the complexity of $\FF$. First, we find an upper bound on
the transversal number $\tau(\FF)$ of $\FF$ in terms of $n$
and the independence number~$\nu(\FF)$. This question is motivated by a problem of Gr\"unbaum \cite{DGrK63}.
Our bound $\tau(\FF) \leq2^n\binom{2n}{n}(\nlogn)\nu(\FF)$
is exponential in $n$, 
an improvement from the previously known bound of
Kim, Nakprasit, Pelsmajer and Skokan \cite{KNPS}, which was of order $n^n$.
By a lower bound, we show that the right order of magnitude is exponential in $n$.

Next, we consider another measure of complexity,
the Vapnik--\v{C}ervonenkis dimension of $\FF$. 
We prove that~$\vcdim(\FF)\le 3$ if $n=2$
and is infinite for some $\FF$ if $n\geq 3$. This settles a conjecture of G\"unbaum \cite{Gr75}: 
Show that the maximum dual VC-dimension of a family of positive homothets of a given convex body $K$ in 
$\Re^n$ is $n+1$. This conjecture was disproved by Naiman and Wynn \cite{NW93}
who constructed a counterexample of dual VC-dimension $\left\lfloor\frac{3n}{2}\right\rfloor$.
Our result implies that no upper bound exists.
\end{abstract}

\maketitle

\section{Definitions and Results}

A \emph{convex body} in~$\Ren$ is a compact convex set with non-empty interior.
A \emph{positive homothet} of a set~$S\subseteq\Ren$
is a set of the form~$\lambda S+x$, where~$\lambda>0$ and~$x\in\Ren$. 
The cardinality, closure, convex hull and volume of $S$ are denoted as
$\card(S), \cl(S),\conv(S)$ and $\vol(S)$, respectively.
The origin of~$\Ren$ is denoted~$o$.

Let $\FF$ be a family of positive homothets (or translates) of a given convex body~$K$ in $\Ren$.
In this note we study two approaches to measuring the complexity of $\FF$.

First, we bound the transversal number $\tau(\FF)$ in terms of the dimension~$n$
and the independence number~$\nu(\FF)$.
The \emph{transversal number}~$\tau(\FF)$ of a family of sets~$\FF$ is defined as
\[ \tau(\FF) = \min\ \{\card(S) \st
    \text{$S\cap F\ne\emptyset$ for all $F\in\FF$}\} . \]
The \emph{independence number}~$\nu(\FF)$ of~$\FF$ is defined as
\[ \nu(\FF) = \max\ \{\card(S) \st
    \text{$S\subseteq\FF$ and $S$ is pairwise disjoint}\} . \]
Clearly~$\nu(\FF)\le\tau(\FF)$. The problem of finding an inequality in the reverse direction originates
in the following question of Gr\"unbaum \cite{DGrK63}: Is it true that $\nu(\FF)=1$ implies $\tau(\FF)\leq 3$
for any family~$\FF$ of translates of a convex body in~$\Re^2$?
Karasev \cite{Kar00} proved the affirmative answer.
One of the main results of \cite{KNPS} by Kim, Nakprasit, Pelsmajer and Skokan
is that in~$\Ren$ we have~$\tau(\FF) \leq 2^{n-1}n^n\nu(\FF)$.
We improve the dependence on~$n$ to exponential.

\begin{thm}\label{thm:ubound}
	Let $K\subseteq\Ren$ be a convex body and
	$\FF$ a family of positive homothets of~$K$.
	Then
	\begin{multline*}
		 \nu(\FF)\leq\tau(\FF)\leq\frac{\vol(2K-K)}{\vol(K)}(\nlogn)\nu(\FF)
		 \\
		 \le\begin{cases}
		     3^n(\nlogn)\nu(\FF) & \text{if $K=-K$,} \\
		     2^n\binom{2n}{n}(\nlogn)\nu(\FF) & \text{otherwise.}
		 \end{cases}
	\end{multline*}
\end{thm}

The following proposition
shows that an exponential bound is the best possible,
even when~$\FF$ contains only translates of~$K$.

\begin{prop}\label{prop:lbound}
	For sufficiently large~$n$, there is a convex body~$K$ in~$\Ren$
	and a family~$\FF$ of translates of $K$
	such that $\tau(\FF)\geq\frac12(1.058)^n\nu(\FF)$.
\end{prop}

Our second approach is to investigate the \emph{VC-dimension} of a family $\FF$
of positive homothets (or translates) of a convex body $K$.
This combinatorial measure of complexity was introduced
by Vapnik and \v{C}ervonenkis \cite{VCru},
and is defined as
\[ \vcdim(\FF) = \sup\ \{\card(X) \st \text{$\FF$ shatters $X$}\} , \]
where a set system~$\FF$ is said to \emph{shatter} a set of points~$X$
if for every subset~$X'\subseteq X$, there exists a set~$F\in\FF$
such that~$X\cap F = X'$.
Note that if there is no upper bound on the sizes of sets shattered by~$\FF$,
then this definition yields~$\vcdim(\FF) = \infty$.

Our main motivation in studying the VC-dimension is
its involvement in upper bounds on transversal numbers
(see the Epsilon Net Theorem of Haussler and Welzl \cite{HW}
and Corollary~10.2.7 of \cite{Mat02})
and related phenomena (see \cite{Mat04}, for example).
We show, however, that~$\vcdim(\FF)$ is bounded 
from above only in dimension two.

\begin{thm}\label{thm:vcdim}
    If~$K\subseteq\Re^2$ is a convex body
    and~$\FF$ is a family of positive homothets of~$K$,
    then~$\vcdim(\FF) \le 3$.
\end{thm}

\begin{ex}\label{ex:vcdim}
    We construct a convex body~$K\subseteq\Re^3$
    and a countable family~$\FF$ of translates of~$K$
    such that~$\vcdim(\FF)=\infty$.
\end{ex}

This example can, of course, be embedded in~$\Re^n$ for~$n>3$ as well.

Example~\ref{ex:vcdim} also settles a conjecture of Gr\"unbaum
on dual VC-dimension (see Section~10.3 of \cite{Mat02} for this notion).
He showed~\cite{Gr75} that
if $\FF$ is a family of positive homothets of a convex body in $\Re^2$,
then $\vcdim(\FF^\ast) \le 3$,
and conjectured (point~(7) on p.~21 of~\cite{Gr75})
the upper bound $\vcdim(\FF^\ast) \le n+1$ for such families in $\Ren$.
(Gr\"unbaum uses a different terminology: instead of dual VC-dimension, 
he writes ``the maximal number of sets in independent families'', where 
``independence'' is \emph{not} as we defined above.)
Naiman and Wynn~\cite{NW93} disproved this conjecture by giving an example
with $\vcdim(\FF^\ast) = \left\lfloor\frac{3n}{2}\right\rfloor$;
our example shows that no upper bound exists,
since $\vcdim(\FF) < 2^{\vcdim(\FF^\ast)+1}$ (\cite{Mat02}, Lemma~10.3.4).

\begin{cor}\label{cor:dualvcdim}
    There is a convex body~$K\subseteq\Re^3$
    and a countable family~$\FF$ of translates of~$K$
    such that~$\vcdim(\FF^\ast)=\infty$.
\end{cor}

The construction of example~\ref{ex:vcdim} shares some principles with
the constructions given in \cite{HM} and in Theorem~2.9 of~\cite{GPW}
to show that certain Helly-type and Hadwiger-type theorems
for line transversals of families of translates of a convex set
in the plane do not generalize to $\Re^3$.
These examples and ours show that, in some sense,
translates of a convex set in $\Re^3$ may form set systems of high complexity. 
They also suggest that finding good bounds for the transversal numbers
of such families is a difficult task.

In Section~\ref{sec:ubound},
we prove Theorem~\ref{thm:ubound} and Proposition~\ref{prop:lbound}.
In Section~\ref{sec:vcdim},
we prove Theorem~\ref{thm:vcdim} and construct Example~\ref{ex:vcdim}.

\section{Transversal and Independence Numbers of Positive Homothets}\label{sec:ubound}

Let $K$ and $L$ be convex bodies in $\Ren$. Let $N(K,L)$ denote the \emph{covering number} of $K$ by $L$; that is, 
the smallest number of translates of $L$ required to cover $K$.
\begin{thm}[Rogers \cite{R57} , Rogers--Zong \cite{RZ}]\label{thm:RZ}
	Let $K,L\subset\Ren$ be convex sets. Then
	\[
	  N(K,L)\leq\frac{\vol(K-L)}{\vol(L)}(\nlogn).
	\]
\end{thm}

\begin{proof}[Proof of Theorem~\ref{thm:ubound}]

	First, we prove the theorem in the case when $\FF$ consists of translates of $K$ only.
	Let $\{K_1, K_2,\dots,K_\ell\}$ be a maximal set of independent (i.e., pairwise disjoint)
	elements of $\FF$. Clearly, $\ell\leq\nu(\FF)$.
	Let $\FF_1=\{F\in\FF \st F\cap K_1\neq\emptyset\}$, and for $i=2,\dots,\ell$ let
	\[\FF_i=\left\{F\in\FF\setminus\bigcup_{j=1}^{i-1}\FF_j \st F\cap K_i\neq\emptyset\right\}.\]
	We will construct a transversal $T_i$ for each $\FF_i$.

        It is easy to show that, for any set~$S\subseteq\Ren$,
        \[ S-K = \{x\in\Ren\st (K+x)\cap S \ne \emptyset\} \text{ .} \]
        An immediate consequence is that if~$K_i-K \subseteq T_i-K$,
        then~$T_i$ is a transversal of~$\FF_i$.
	By Theorem~\ref{thm:RZ}, for each $i$, 
	there is such a set $T_i$ with 
	\begin{align*}
	  \card(T_i) &\le \frac{\vol(K_i-K+K)}{\vol(-K)}(\nlogn) \\
		&= \frac{\vol(2K-K)}{\vol(K)}(\nlogn) \\
		&\le \begin{cases}
		     3^n(\nlogn) & \text{if $K=-K$,} \\
		     2^n\binom{2n}{n}(\nlogn) & \text{otherwise.}
                     \end{cases}
	\end{align*}
	The last inequality for the non-symmetric case follows from the Rogers--Shephard inequality \cite{RS}. 
	Hence, $T=\mathop\cup\limits_{i=1}^\ell T_i$ is a transversal of $\FF$ 
	of cardinality bounded from above as stated in the theorem.
	
	The proof of the case when $\FF$ contains finitely many positive homothets of $K$
	follows from an argument given in \cite{KNPS},
        which we repeat here.
	First, assume that $\inf\ \{\lambda \st \lambda K+x\in\FF\} >0$.
	Let $\varepsilon$ be a positive number, to be specified later.
	We say that $\lambda K+x$ is a \emph{small} member of
        a subset $\mathcal A\subseteq\FF$ if 
	\[ \lambda
            < (1+\varepsilon)\inf\ \{\mu \st \mu K+x\in\mathcal A\}. \]
	Let $F_1$ be a small element of $\FF$, and let
	$\FF_1=\left\{F\in\FF \st F\cap F_1\neq\emptyset\right\}.$
	Next, for each $i=2,3,\dots,\ell$ inductively, let $F_i$ be 
	a small element in $\FF\setminus\mathop\cup\limits_{j=1}^{i-1}\FF_j$, and let
	\[\FF_i=\left\{F\in\FF\setminus\bigcup_{j=1}^{i-1}\FF_j \st F\cap F_i\neq\emptyset\right\}.\]
	Let $\lambda_i=\inf\ \{\lambda \st \lambda K+x\in\FF_i\}$.
	By assumption, $\lambda_i>0$. Our inductive procedure of defining $F_i, \FF_i$ and $\lambda_i$
	will terminate with $\ell\le\nu(\FF)$.

	Now, for each $F\in\FF_i$, choose a point $z$ in $F\cap F_i$, and shrink $F$ with center $z$
	to obtain a translate of $\lambda_i K$. 
	The shrunk copy of $F$ is clearly contained in $F$. Let $\FF_i'$ be the family of these shrunk copies.
	Now, $\FF_i'$ contains only translates of $\lambda_i K$,
        any transversal of $\FF_i'$ is a transversal of $\FF_i$,
        and each member of $\FF_i'$ intersects $F_i$.
        Thus if~$F_i-\lambda_i K \subseteq T_i-\lambda_i K$,
        then~$T_i$ is a transversal of~$\FF_i$.
        Theorem~\ref{thm:RZ} yields such a set $T_i$ with cardinality
	\[ \card(T_i) \le \frac{\vol((1+\varepsilon)\lambda_i K- \lambda_i K + \lambda_i K)}{\vol(-\lambda_i K)}(\nlogn). \]
	Since $\card(T_i)$ is an integer,
        choosing a sufficiently small $\varepsilon$ provides the right bound.

	Finally, we sketch the additions necessary to handle 
	the case when\\ $\inf\ \{\lambda \st \lambda K+x\in\FF\} =0$,
        a case not considered in~\cite{KNPS}.
        Let $(\delta_m)_{m=1}^\infty$ be a sequence of positive real
        numbers with $\delta_m\downarrow 0$.
        For every $m\in\Ze^+$ we define
	$\FF^m=\{\lambda K +x \in\FF \st \lambda>\delta_m\}$. Using the previous proof, we obtain a
	transversal $T^m = \{t^m_1,\dotsc,t^m_k\}$ of $\FF^m$ for each $m$,
        where $k$ is the desired bound.
Now, choose some $G_1\in\FF$.
By the pigeonhole principle,
there is an $i\in\{1,\dotsc,k\}$
with $t^m_i\in G_1$ for infinitely many $m$;
assume $i=1$.
Passing to a subsequence of $(T^m)_{m=1}^\infty$,
we may further assume that $t^m_1\to t_1\in G_1$.
If $\{t_1\}$ is not a transversal of $\FF$,
choose $G_2\in\FF$ with $t_1\notin G_2$;
passing to a further subsequence of $(T^m)_{m=1}^\infty$,
we may assume that $t^m_2\to t_2 \in G_2$.
If $\{t_1,t_2\}$ is not a transversal of $\FF$,
continue in this manner,
obtaining eventually a transversal of $\FF$.
\end{proof}

For the proof of Proposition~\ref{prop:lbound},
we need the following definition.
A set $S\subseteq\Ren$ is called \emph{strictly antipodal} if,
for any two points $x_1$ and $x_2$ in $S$,
there exists a hyperplane~$H$ through~$o$
such that~$H+x_1$ and~$H+x_2$ support~$S$
and~$(H+x_1)\cap S = \{x_1\}$ and~$(H+x_2)\cap S = \{x_2\}$. 
For more on this notion, see~\cite{Gr63}.

\begin{proof}[Proof of Proposition~\ref{prop:lbound}]
	First, we show that if $S$ is a strictly antipodal set
	then $\FF=\{K+s \st s\in S\}$, where $K=\conv(S)$, is a family of 
	pairwise touching translates of $K$,
        and no three members of $\FF$ have a point in common.
	We may assume that $o\in K$.
	Let $x_1,x_2$ be two distinct points in $S$.
        Clearly, $x_1+x_2\in (K+x_1)\cap(K+x_2)$.
        On the other hand, if $H$ is a hyperplane
        as in the definition of strict antipodality,
        then $H'=H+x_1+x_2$ separates $K+x_1$ and $K+x_2$.
        Moreover, $(K+x_1)\cap H' = (K+x_2)\cap H' = \{x_1+x_2\}$.
	So, $K+x_1$ and $K+x_2$ touch each other.
        We need to show that for any $x_3\in S\setminus\{x_1,x_2\}$, 
	we have that $K+x_3$ does not contain $x_1+x_2$. Suppose it does.
        Then $x_1+x_2$ is a common point of
	$K+x_1$ and $K+x_3$, hence, by the previous argument,
        $x_1+x_2=x_1+x_3$, so $x_2=x_3$, a contradiction.

	On the other hand, F\"uredi, Lagarias and Morgan 
	(Theorem 2.4. in \cite{FLM})
        give a construction, for sufficiently large~$n$,
        of a symmetric strictly convex body $K$ and a finite set
	$S$ in $\Ren$ with the property that 
	any two translates of $K$ in the family 
	$\{s+K : s\in S\}$ touch each other, moreover 
	$\card(S) \geq (1.02)^n$. It follows that 
	$S$ is a strictly antipodal set.
	Later, Swanepoel observerd (Theorem 2 in Section 2.2, \cite{Sw04}) 
	that a better bound, $\card(S) \geq (1.058)^n$ 
	follows from the proof in \cite{FLM}.
        Thus, for the resulting $\FF$ we have $\nu(\FF)=1$ 
	and $\tau(\FF)\geq \frac{1}{2}\card(\FF)=\frac{1}{2}(1.058)^n$.
\end{proof}

\section{VC-Dimension of Positive Homothets}\label{sec:vcdim}

\begin{proof}[Proof of Theorem~\ref{thm:vcdim}]
Let~$\FF$ be a family of positive homothets of a convex body~$K\subseteq\Re^2$.
Suppose, for contradiction, that~$\FF$ shatters some set of four points,
say,~$X = \{x_1,x_2,x_3,x_4\}$.

Case~1: One of the points of~$X$ is in the convex hull of the other three,
say,~$x_1\in\conv(\{x_2,x_3,x_4\})$.
By hypothesis, there is an~$F\in\FF$
such that~$X\cap F = \{x_2,x_3,x_4\}$.
But since~$F$ is convex, it follows that~$x_1\in F$,
which is a contradiction.

\begin{figure}[t]
    \begin{minipage}{0.49\textwidth}
        \begin{center}\input{quad-fig.tex}\end{center}
        \par\caption{Theorem~\ref{thm:vcdim}, Case~2.}\label{fig:quad}
    \end{minipage}
    \hfill
    \begin{minipage}{0.49\textwidth}
        \begin{center}\input{conv-fig.tex}\end{center}
        \par\caption{Why~$p\notin A$.}\label{fig:conv}
    \end{minipage}
\end{figure}

Case~2: The points of~$X$ are in convex position,
forming the vertices of a convex quadrilateral in,
say, the order~$x_1x_2x_3x_4$.
(See Figure~\ref{fig:quad}.)
Without loss of generality,~$X\cap K = \{x_1,x_3\}$
and~$X\cap TK = \{x_2,x_4\}$,
where~$T\colon\Re^2\to\Re^2, Tx = \lambda x + t$
is a homothety with ratio~$\lambda\ge 1$.

First suppose~$\lambda > 1$.  Let
\[ p = \frac{1}{1-\lambda} t \]
be the centre of the homothety~$T$.
If~$p$ is in the (closed) region~$A$ shown in Figure~\ref{fig:quad},
then~$x_2\in\conv(\{x_1,x_3,p\})$.
On the other hand, $T^{-1}x_2$ is a convex combination of~$p$ and~$x_2$;
thus $x_2\in\conv(\{x_1,x_3,T^{-1}x_2\})$.
(See Figure~\ref{fig:conv}.)
But~$\{x_1,x_3,T^{-1}x_2\} \subseteq K$,
so by convexity,~$x_2\in K$, a contradiction.

Similarly, if~$p\in B$ then~$x_4\in\conv(\{x_1,x_3,T^{-1}x_4\})\subseteq K$;
if~$p\in C\cup D$ then~$x_3\in\conv(\{x_2,x_4,Tx_3\})\subseteq TK$;
and if~$p\in D\cup E$ then~$x_1\in\conv(\{x_2,x_4,Tx_1\})\subseteq TK$.
In all cases we obtain a contradiction.

The case~$\lambda = 1$, when~$T$ is a translation,
succumbs to essentially the same argument, with~$p$ an ideal point
corresponding to the direction of the translation.
We omit the details.
\end{proof}

\begin{figure}
    \begin{center}\input{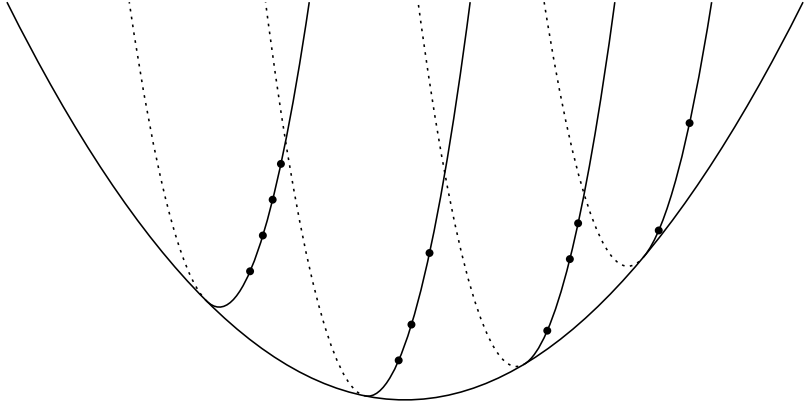}\end{center}
    \par\caption{The paraboloid $z=x^2+y^2$ and a few sections of it.}
    \label{fig:paraboloid}
\end{figure}

\begin{proof}[Construction of Example \ref{ex:vcdim}]
To illustrate the ideas of the construction,
we first sketch how to construct, for any~$M\in\N$,
a convex body~$K$ whose translates shatter a set of~$M$ points.

The sections of the paraboloid~$z = x^2 + y^2$
by planes parallel to the $yz$-plane
are all translates of the same parabola.
(See Figure~\ref{fig:paraboloid}.)
Choose some~$2^M$ of these sections
and some set~$X$ of~$M$ points on one of them.
Each section contains a translated copy of~$X$;
assign a subset to each section, take that subset of its copy of~$X$,
and let~$K$ be the convex hull of the points in these subsets of copies.
The translates of~$K$ then shatter~$X$,
since an appropriate translation will superimpose
the section corresponding to any desired subset
on the section containing~$X$.

Now, we present Example~\ref{ex:vcdim}.
Let~$\mathcal E$ be the family of all finite subsets of~$\N$,
and let~$E\colon\N\to\mathcal E$ be a bijection.
Set
\[ A = \{(m,n)\in\N^2\st m\in E(n)\} \text{ .} \]
For~$m,n\in\N$,
let~$u_m = (\frac1m,0,\frac1{m^2})$
and~$v_n = (0,\frac1n,\frac1{n^2})$,
and define
\[ p\colon\N^2\to\Re^3 \text{ ,\quad}
    p(m,n) = u_m + v_n \text{ .} \]
Let~$K = \conv(\cl(p(A))$
and~$\FF = \{K - v_n \st n\in\N\}$.
We claim that~$\vcdim(\FF) = \infty$.

Let~$P\subseteq\Re^3$ be the paraboloid with equation~$z = x^2 + y^2$.
Since~$P$ is the boundary of a strictly convex set,
$P\cap\conv(S) = S$ for any~$S\subseteq P$.
Since~$p(\N^2)$ is a discrete set,
$p(\N^2)\cap\cl(S) = S$ for any~$S\subseteq p(\N^2)$.
So if~$T \subseteq p(\N^2)$, then
\[ T\cap K = T\cap p(\N^2) \cap P \cap K
    = T\cap p(\N^2) \cap \cl(p(A)) = T\cap p(A) \text{ .} \]
Now, let~$M\in\N$,
$X = \{u_1,\dotsc,u_M\}$, and $X'\subseteq X$.
Let~$n\in\N$ be such that~$X' = \{u_m \st m\in E(n)\}$.
Then
\[ (X + v_n)\cap K = (X + v_n)\cap p(A) = X' + v_n \text{ ,} \]
that is,~$X\cap (K - v_n) = X'$.
Thus~$\FF$ shatters~$X$, so~$\vcdim(\FF) \ge M$.
\end{proof}

\section*{Acknowledgements}
The authors are grateful to Nicole Tomczak-Jaegermann
and Alexander Litvak for their support and encouragement,
and we thank the University of Alberta.
We thank Leonard Schulman who, upon learning of our example~\ref{ex:vcdim},
brought the question asked by Gr\"unbaum
in \cite{Gr75} to our attention and thus put our result in context.
The first named author holds a Postdoctoral Fellowship 
of the Pacific Institute for the Mathematical Sciences
at the University of Alberta,
and the second named author was supported by a Canada Graduate Scholarship
of the Natural Sciences and Engineering Research Council of Canada.
We thank them as well.

\section{A note}
After the publication of the paper, Konrad Swanepoel brought the following to our attention:
In Lemma 9.11.2 of \cite{Bor04} (proved by I. Talata in \cite{Talata})
an explicit construction of an $o$-symmetric strictly convex smooth body is given 
with $\sqrt[3]{3^n}/3$ pairwise touching translates.
That changes the bound in Proposition~\ref{prop:lbound} to
$\tau(\FF)\geq\frac{\sqrt[3]{3^n}}{6}\nu(\FF)$.

\bibliographystyle{plain}
\bibliography{biblio}

\begin{thebibliography}{10}

\bibitem{Bor04}
K{\'a}roly B{\"o}r{\"o}czky, Jr.
\newblock {\em Finite packing and covering}, volume 154 of {\em Cambridge
  Tracts in Mathematics}.
\newblock Cambridge University Press, Cambridge, 2004.

\bibitem{DGrK63}
Ludwig Danzer, Branko Gr{\"u}nbaum, and Victor Klee.
\newblock Helly's theorem and its relatives.
\newblock In {\em Proc. {S}ympos. {P}ure {M}ath., {V}ol. {VII}}, pages
  101--180. Amer. Math. Soc., Providence, R.I., 1963.

\bibitem{FLM}
Z.~F{\"u}redi, J.~C. Lagarias, and F.~Morgan.
\newblock Singularities of minimal surfaces and networks and related extremal
  problems in {M}inkowski space.
\newblock In {\em Discrete and computational geometry ({N}ew {B}runswick, {NJ},
  1989/1990)}, volume~6 of {\em DIMACS Ser. Discrete Math. Theoret. Comput.
  Sci.}, pages 95--109. Amer. Math. Soc., Providence, RI, 1991.

\bibitem{GPW}
Jacob~E. Goodman, Richard Pollack, and Rephael Wenger.
\newblock Geometric transversal theory.
\newblock In {\em New trends in discrete and computational geometry}, volume~10
  of {\em Algorithms Combin.}, pages 163--198. Springer, Berlin, 1993.

\bibitem{Gr63}
Branko Gr{\"u}nbaum.
\newblock Strictly antipodal sets.
\newblock {\em Israel J. Math.}, 1:5--10, 1963.

\bibitem{Gr75}
Branko Gr{\"u}nbaum.
\newblock Venn diagrams and independent families of sets.
\newblock {\em Math. Mag.}, 48:12--23, 1975.

\bibitem{HW}
David Haussler and Emo Welzl.
\newblock {$\epsilon$}-nets and simplex range queries.
\newblock {\em Discrete Comput. Geom.}, 2(2):127--151, 1987.

\bibitem{HM}
Andreas Holmsen and Ji{\v{r}}{\'{\i}} Matou{\v{s}}ek.
\newblock No {H}elly theorem for stabbing translates by lines in {$\Bbb R\sp
  3$}.
\newblock {\em Discrete Comput. Geom.}, 31(3):405--410, 2004.

\bibitem{Kar00}
R.~N. Karasev.
\newblock Transversals for families of translates of a two-dimensional convex
  compact set.
\newblock {\em Discrete Comput. Geom.}, 24(2-3):345--353, 2000.
\newblock The Branko Gr{\"u}nbaum birthday issue.

\bibitem{KNPS}
Seog-Jin Kim, Kittikorn Nakprasit, Michael~J. Pelsmajer, and Jozef Skokan.
\newblock Transversal numbers of translates of a convex body.
\newblock {\em Discrete Math.}, 306(18):2166--2173, 2006.

\bibitem{Mat02}
Ji{\v{r}}{\'{\i}} Matou{\v{s}}ek.
\newblock {\em Lectures on discrete geometry}, volume 212 of {\em Graduate
  Texts in Mathematics}.
\newblock Springer-Verlag, New York, 2002.

\bibitem{Mat04}
Ji{\v{r}}{\'{\i}} Matou{\v{s}}ek.
\newblock Bounded {VC}-dimension implies a fractional {H}elly theorem.
\newblock {\em Discrete Comput. Geom.}, 31(2):251--255, 2004.

\bibitem{NW93}
Daniel~Q. Naiman and Henry~P. Wynn.
\newblock Independent collections of translates of boxes and a conjecture due
  to {G}r\"unbaum.
\newblock {\em Discrete Comput. Geom.}, 9(1):101--105, 1993.

\bibitem{R57}
C.~A. Rogers.
\newblock A note on coverings.
\newblock {\em Mathematika}, 4:1--6, 1957.

\bibitem{RS}
C.~A. Rogers and G.~C. Shephard.
\newblock The difference body of a convex body.
\newblock {\em Arch. Math.}, 8:220--233, 1957.

\bibitem{RZ}
C.~A. Rogers and C.~Zong.
\newblock Covering convex bodies by translates of convex bodies.
\newblock {\em Mathematika}, 44(1):215--218, 1997.

\bibitem{Sw04}
Konrad~J. Swanepoel.
\newblock Equilateral sets in finite-dimensional normed spaces.
\newblock In {\em Seminar of {M}athematical {A}nalysis}, volume~71 of {\em
  Colecc. Abierta}, pages 195--237. Univ. Sevilla Secr. Publ., Seville, 2004.

\bibitem{Talata}
Istv\'an Talata.
\newblock On equilateral dimensions of strictly convex bodies.
\newblock {\em in preparation}.

\bibitem{VCru}
V.~N. Vapnik and A.~Ja. {\v{C}}ervonenkis.
\newblock On the uniform convergence of relative frequencies of events to their
  probabilities.
\newblock {\em Dokl. Akad. Nauk SSSR, 181}, 4:781ff., 1968.
\newblock In Russian; English translation in Theor.\ Probab.\ Appl.,
  16:264--280 (1971).

\end{thebibliography}
\end{document}